\documentclass[11pt,a4paper,reqno]{amsart}
\usepackage{amsmath,amsfonts,amssymb,amsthm}
\usepackage[english]{babel}
\usepackage{xcolor}
\usepackage{enumerate}

\newtheorem{theorem}{Theorem}

\newtheorem{lemma}{Lemma}

\newtheorem{corollary}{Corollary}

\newtheorem{theo}{Theorem}
\newtheorem{coro}{Corollary}

\newtheorem{lem}{Lemma}

\theoremstyle{definition}

\theoremstyle{remark}
\newtheorem*{remark}{Remark}

\def\R{\mathbb{R}}

\title{Hausdorff measures, dyadic approximations and Dobi\'nski set}

\author{Alberto Dayan, Jos\'e L. Fern\'andez, Mar\'ia J. Gonz\'alez }

\subjclass[2010]{11K55, 28A78, 30C85}
\keywords{Hausdorff measures, Mass Transference Principle, dyadic approximation, Dobi\'nski set, logarithmic capacity.}
\date{\today}

\begin{document}

	\begin{abstract}
	Dobi\'nski set $\mathcal{D}$ is an exceptional set for a certain infinite product identity, whose points are characterized as having exceedingly good approximations by dyadic rationals.	We study the Hausdorff dimension and logarithmic measure of  $\mathcal{D}$ by means of the Mass Transference Principle and by the construction of certain  appropriate Cantor-like sets, termed willow sets,  contained in $\mathcal{D}$.
	\end{abstract}
	
	\maketitle

	\parskip=1.5mm

	\section{G. Dobi\'nski}
Two formulas in the mathematical literature are named after our G. Dobi\'nski.\footnote{Little we know  about G. Dobi\'nski.  He  published several notes, \cite{dob2}, \cite{dob1}, \cite{dob3}, \cite{dob4}, \cite{dob5} and \cite{dob6} in the Archiv der Mathematik und Physik between 1876 and 1879. The Archiv der Mathematick und Physik proclaims that \textit{mit besonderer R{\"{u}}cksicht auf die Bed{\"{u}}rfnisse der Lehrer an h{\"{o}}heren  Unterrichtsanstalten}.  His name appears misspelled as Dobiciecki in \cite{dob1}, but it is corrected in the index of that issue. Dobi\'nski gives as his affiliation  \textit{Techniker in Warschau}, and  \textit{Techniker an der Eisenbahn in Kutno} or \textit{Techniker der Warschau-Bromberger Eisenbahn}. } Two mathematical eponyms out of a handful of papers; not a bad batting average.

The content of Dobi\'nski's paper \cite{dob2} is the presentation of the following  formula for the   $n^{\rm th}$ Bell number:
$$(\star) \qquad B_n=\frac{1}{e}\sum_{k=0}^\infty \frac{k^n}{k!}\, , \quad \mbox{for each $n \ge 0$}\,,$$
where $0^0=1$, so that $B_0=1$.

For $n \ge 1$, the Bell number $B_n$ counts the number of distinct partitions of the set $\{1, 2, \ldots, n\}$ into non empty parts, while $B_0=1$ is a convenient convention.
Classifying partitions by the elements in the block containing $n{+}1$ one sees that the Bell numbers satisfy the recurrence relation
$$B_{n+1}=\sum_{k=0}^{n} \binom{n}{k} B_k\, , \quad \mbox{for $n \ge0$}\,.$$

Actually, what Dobi\'nski shows in \cite{dob2} is that the numbers $D_n$ given by
$$
D_n=\sum_{k=0}^\infty \frac{k^n}{k!}\, , \quad \mbox{$n \ge 1$}
$$
satisfy\footnote{Dobi\'nski verifies  the recurrence relation only up to $n=7$, but once you know that it is true, it is easy to check its general validity by induction.} the recurrence relation
$$D_{n{+}1}=\sum_{k=0}^{n} \binom{n}{k} D_k\, , \quad \mbox{for $n \ge1$}\,,$$
and, thus, since $D_0=e$, we have that $D_n=eB_n$, for $n \ge 1$.

Further, in \cite{dob4}, Dobi\'nski shows that
$$
\sum_{n=0}^\infty \frac{D_n}{n!} z^n=e^{e^z}\,,$$
and thus that $e^{e^z-1}$ is the exponential generating function of the Bell numbers.

\medskip

No motivation for the consideration of the sums $D_n$, and no connection with combinatorics are presented in Dobi\'nski's paper. It  appears that  it was  Rota in \cite{rota} who firts pointed out the connection of the numbers $D_n$ with the Bell numbers, termed $(\star)$ Dobin\'ski formula and called attention to its relevance. There are Dobi\'nski type formulas in many different contexts, nowadays.

\

Let us turn now to the second of the formulas named after Dobi\'nski. In \cite{dob1}, and later on  in \cite{dob3}, Dobi\'nski\footnote{Dobi\'nski returns to the subject of obtaining formulas for some infinite products and sums of series by telescoping arguments in \cite{dob5} and \cite{dob6}, and presents the Euler-Vi\`ete formula for $\sin x/x$ as infinite product of cosines or the formula
$$\sum_{n=1}^\infty \arctan(2/n^2)=3\pi/4\,,$$ with credit to none other than Beltrami who proposed the formula as a \textit{quistione} in \textit{Giornale de Matematiche}, vol. 5, page 189, and to  Grunert, who  gave a proof in a note by the \textit{herausgeber}, in his \textit{Archiv der Mathematik und Physik}, vol 47, page 362. } proposes the following infinite product identity:
\begin{equation}\label{eq:dob identity1}
\prod_{j=0}^\infty \big(\tan 2^j \pi x\big)^{1/2^j}=(2 \sin \pi x)^2\, ,
\end{equation} which is obtained  by doubling the angle reiteratively in  the basic identity
$$
\tan \pi x=\frac{2 \sin^2 \pi x}{\sin 2 \pi x}\, .$$
Formula \eqref{eq:dob identity1} is purportedly valid for all $x\in \mathbb{R}\setminus \widetilde{\mathcal{Q}}$, where $\widetilde{\mathcal{Q}}$ is the set of dyadic rationals, i.e., real numbers $x$ of the form  $x= h/2^k$ for integer $k \ge 0$ and integer $h$.

For $x \notin \widetilde{\mathcal{Q}}$ none of the numbers $\pi x, 2\pi x, 4\pi x, \ldots$ is an odd multiple of $\pi/2$ and thus the $\tan$'s in the formula \eqref{eq:dob identity1} are well defined.

There is still the issue as to what is the proper definition of $(\tan 2^j \pi x)^{1/2^j}$ when $(\tan 2^j \pi x)$ is negative. We will not  elaborate upon this point, which is carefully dealt with in \cite{agnewwalker}, opt to bypass it and  consider instead the identity
\begin{equation}\label{eq:dob identity}
\prod_{j=0}^\infty \big|\tan 2^j \pi x\big|^{1/2^j}=(2 \sin \pi x)^2\, ,
\end{equation}
which we term \textit{Dobi\'nski identity}.

But,  alas,  as pointed out and discussed at length by   Agnew  and  Walker in \cite{agnewwalker}, the identity \eqref{eq:dob identity} is not valid for all $x \in \R\setminus \widetilde{\mathcal{Q}}$.

Let us restrict ourselves to $x$ in the interval $[0,1]$. For $x \in [0,1]$ and integer $n \ge 1$, one has
$$\prod_{j=0}^n \big|\tan 2^j \pi x\big|^{1/2^j}=\frac{2^{\sum_{j=0}^n 2^{-j}} \sin^2 \pi x}{|\sin 2^{n{+}1} \pi x|^{1/2^n}}\, .$$

As $n \to \infty$, the numerator of this expression tends to $4 \sin^2 \pi x$, but the denominator does not necessarily tend to 1.

\

Let us denote $\mathcal{Q}=\widetilde{\mathcal{Q}}\cap [0,1]$ and $\mathcal{I}\triangleq [0,1]\setminus \mathcal{Q}$. For $x \in \mathcal{I}$ write its (unique) binary expansion as
$$x = \sum_{k=1}^\infty \epsilon_j(x)/2^j\, , \quad \mbox{with $\epsilon_j(x)\in \{0,1\}$}\,, $$
For $x \in (0,1)\cap \mathcal{Q}$ we have two binary expansions, one with digit  1 at a certain position and then digit 0 afterwards and another one with 0 at a certain position and then digit 1 afterwards. For $x=0$ and $x=1$ the binary expansion is unique.

For each $x  \in \mathcal{I}$ and $n \ge 1$,  we define $z_n(x)$ as the length of the longest constant  string of digits $\epsilon_k$ after $\epsilon_n$, so that
$$\epsilon_{n{+}1}(x)=\dots=\epsilon_{n{+}z_n}(x),$$
but $\epsilon_{n{+}z_n}(x)\neq \epsilon_{n{+}z_n{+}1}(x)$.

For $x \in \mathcal{Q}$ we extend the definition of $z_n(x)$, which is now $z_n(x)=\infty$, for some $n$ onwards.

\medskip

Then, see the discussion in \cite{agnewwalker}, we have that
$$
\lim_{n \to \infty} |\sin 2^{n{+}1} \pi x|^{1/2^n}=1 \quad \mbox{if and only if} \quad \lim_{n \to \infty} \frac{z_n(x)}{2^n}=0\, .$$

Thus Dobi\'nski 's identity \eqref{eq:dob identity} does \emph{not} hold precisely for those  $x$ such that
$$\limsup_{n \to \infty} \frac{z_n(x)}{2^n}>0\,.$$

\medskip

Thus the exceptional set in Dobi\'nski's identity consists of those $x$ which have infinitely many exceedingly
 long (in terms of its starting position) strings of zeros or ones in its binary expansion.


We define now the Dobi\'nski set $\mathcal{D}$ as the exceptional set in Dobi\'nski's identity given by
$$\mathcal{D}=\Big\{x \in[0,1]: \limsup_{n \to \infty} \frac{z_n(x)}{2^n}> 0\Big\}\,$$
For $x \in \mathcal{Q}$, the $\limsup$ above is actually $\infty$.

\

\medskip

For $n \ge 1$, let us define  $$\mathcal{Q}_n=\Big\{\frac{k}{2^n}: 0 \le k \le 2^n\Big\}\, .$$
Thus $\mathcal{Q}_n\subset \mathcal{Q}_{n{+}1}$, for $n \ge 1$,  and $\bigcup_{n \ge 1} \mathcal{Q}_n=\mathcal{Q}$.

For $x\in\mathcal{I}$ and $n \ge 1$, let us denote
$$S_n(x)=\sum_{j=1}^n \epsilon_j(x)/2^j\, .$$
and by $P_n(x)$ the point in $\mathcal{Q}_n$ which is closest to $x$.
Observe that if $\epsilon_{n{+}1}(x)=0$, then $P_n(x)=S_n(x)<x$, while if $\epsilon_{n{+}1}(x)=1$, then $P_n(x)=S_n(x)+1/2^n>x$.

Besides, for $x \in \mathcal{I}$ and $n \ge 1$, one has that
$$\frac{1}{2^{n{+}z_n(x){+}1}}\le |x-P_n(x)|\le \frac{1}{2^{n{+}z_n(x)}}\, .$$

Thus the Dobi\'nski set is given by
$$\mathcal{D}=\bigg\{x \in [0,1]: \limsup\frac{-\ln_2 |x-P_n(x)|}{2^n}> 0\bigg\}\, .$$
For $x \in \mathcal{Q}$, the $\limsup$ above is actually $\infty$.
Points in the Dobi\'nski set have exceedingly good approximations by dyadic rationals.

\medskip

If we define $A_{n, k}$ for $k,n\ge 1$ as
$$A_{n, k}=\bigcup\limits_{y\in \mathcal{Q}_n}B(y,2^{- 2^n/k})\, ,$$ the Dobi\'nski set appears as
$$\mathcal{D}=\bigcup_{k\ge 1}\mathcal{D}(k) \quad \mbox{where} \quad \mathcal{D}(k)=\bigcap\limits_{j \ge 1}\bigcup\limits_{n \ge j} A_{n,k}\, ;$$
the points in the Dobi\'nski subset $\mathcal{D}(k)$ are those points which lie in infinitely of the $A_{n,k}$.

\medskip

Our goal in this note is to study the size of the Dobi\'nski set $\mathcal{D}$.

As for its Lebesgue measure, and as pointed out in \cite{agnewwalker}, one has that $|\mathcal{D}|=0$.
This is simply seen as follows. Since $|A_{n,k}| =2^n 2^{-2^n/k}$, the Borel-Cantelli lemma gives us comfortably that $|\mathcal{D}(k)|=0$, for each $k \ge 1$, and consequently that $|\mathcal{D}|=0$.

\

\begin{remark}Let $T:[0,1] \rightarrow [0,1]$ be the transformation $T(x)=2x \mod 1$. Denote the iterates of $T$ by $T^n, n \ge 1,$ and observe that  $T^n(x)=2^n(x-S_n(x))$, for $x \in [0,1]$.

Define $I_n=[0,2^{-2^n})$, for each $n \ge 1$ and let the set $\mathcal{T}$ consists of those $x \in [0,1]$ such that $T^n(x)\in I_n$ for infinitely many $n$.
The orbit under $T$ of a point in $\mathcal{T}$  hits infinitely many times  the rapidly shrinking target $I_n$.

If $x\in \mathcal{T}$, then $|x-S_n(x)|<2^{-n-2^n}$ infinitely many times, and thus, for infinitely many $n$ one has both that $z_n(x)>{2^n}$ and that $\epsilon_{n{+}1}(x)=0$. The points in $\mathcal{T}$ are those points in $\mathcal{D}(1)$ with long strings of zeros.

The points in $\mathcal{D}(1)$ with long strings of ones appear analogously by iterating $x \mapsto 2(1-x) \mod 1$.\end{remark}

\

\noindent \textit{Outline.} Section \ref{sect:MTP} describes the Mass Transference Principle of Beresnevich and Velani, together with some background on Hausdorff measures and dimensions. Section \ref{sect:dyadic approx and dobinski} is devoted to dyadic approximations and the study of the size of the Dobi\'nski set.  An alternative approach via Cantor-like sets, which we call willow sets,  to analyze the size of the Dobi\'nski set is presented in Section \ref{sect:Cantor}. Section \ref{sect:capacity} proposes a question about the logarithmic capacity of the Dobi\'nski set.

\

\noindent \textit{Acknowledgements.} We would like to thank Professors Alicia Cant\'on and Mar\'{\i}a Victoria Meli\'an for some fruitful conversations.\\
The first author was supported by  ERC grant 307179-GFTIPFD.\\
The third author is supported by Plan Nacional I+D grant no. MTM2017-85666-P, Spain. She  would also like to express her gratitude to the IMPAN for the invitation to participate in the Simons Semester. It was during this stay that the final version of this paper has been completed.

\section{ Hausdorff measure and the mass transference principle}\label{sect:MTP}

 \subsection{Hausdorff measures and dimensions} We begin this section by establishing the notation and recalling the basic definitions regarding Hausdorff measures. We refer the reader to \cite{falc} or \cite {matt} for  further background information on this  topic.

In what follows, a \emph{gauge function}  $h\colon[0, +\infty)\to[0, +\infty]$ is a right-continuous, increasing function such that $h(0)=0$. For a set $E\subset \mathbb{R}^m$, and $\delta>0$, consider  coverings of $E$ by sequences of balls $B(x_i,r_i)_{i \ge 1}$ with radii $r_i \leq\delta$, and denote
$$
\mathcal{H}_h^\delta(E)= \inf\bigg\{\sum_{i\ge1}h(r_i):~E \subset \bigcup_{i\ge 1} B(x_i,r_i)~ \textrm{and}~ r_i\leq \delta\bigg\}.
$$
	
The Hausdorff $h$-measure of $E$ denoted by $\mathcal{H}_h (E)$ is defined as
$$
\mathcal{H}_h (E)= \lim_{\delta\rightarrow 0}\mathcal{H}_h^\delta(E).
$$
When $h(r ) = r^s$, and $s>0$, we call $\mathcal{H}_h (E)$  the $s$-dimensional Hausdorff measure of $E$
and write $\mathcal{H}_h (E)=\mathcal{H}_s (E)$. The exponent $s=m$  corresponds to the Lebesgue measure in $\mathbb{R}^m$.

\noindent The case  $h(r)={1}/{(\log 1/x)^s}$, and $s>0$, will also be of special interest for us; the associated Hausdorff measure is denoted by $\mathcal{L}_s (E)$ and called the $s$-logarithmic dimensional Hausdorff measure.

The Hausdorff dimension of $E$ is
$$
\mbox{\rm dim}(E) = \inf\{s>0:\mathcal{H}_s (E)  = 0\}\, ;
$$
while the logarithmic Hausdorff dimension of $E$ is
$$\mathcal{L}{-}\mbox{\rm dim}(E) = \inf\{s>0:\mathcal{L}_s (E)  = 0\}.
$$

\subsection{The Mass Transference Principle (MTP)}

In the context of diophantine approximation, Beresnevich and Velani   study in \cite{BV} the Hausdorff measure of sets  which are defined as lim sup of a sequence of balls $(B_i)_{i \ge 1}$, that is
$$
\lim \sup_{i\rightarrow\infty} B_i \triangleq \bigcap_{j=1}^\infty \bigcup_{k\geq j} B_k
$$
Their  key result,  called the \textit{Mass Transference Principle}, \emph{MTP}, for short,  allows to infer information on $h$-Hausdorff measure from statements on Lebesgue measure.

We only consider the (basic) \emph{MTP} for dimension 1. For a ball  $B=B(x,r)$ in $\R$ and gauge function $h$   we denote by $B^h$ the ball $B(x,h(r))$.

\begin{theorem}[Beresnevich, Velani] \textbf{Mass Transference Principle.}
 Let $(B_i)_{i \ge 1}$ be a sequence of balls in $\mathbb{R}$ with
$\textrm{diam}(B_i)\rightarrow 0$ as $i \rightarrow \infty$.

Let $h$ be a gauge function such that $h(x)/x$ is non increasing, for all $x$ sufficiently close to $0$.

If for any interval $I \subset \R$ one has that
$$
|I \cap \limsup_{i\rightarrow \infty} B_i^h|= |I|\,,
$$
then for any interval $I \subset \R$ one has that
$$
\mathcal{H}_h \big(I \cap \limsup_{i\rightarrow \infty} B_i\big)= \mathcal{H}_h (I).
$$
\end{theorem}

\medskip

We shall appeal later on to the following local version of \emph{MTP}.
\begin{corollary} \label{cor MTP}Let $(B_i)_{i \ge 1}$ be a sequence of balls in $\mathbb{R}$ centered at points in the interval $[0,1]$ with
$|B_i|\rightarrow 0$ as $i \rightarrow \infty$.

Let $h$ be a gauge function such that $h(x)/x$ is non increasing, for all $x$ sufficiently close to $0$.

If
$$
(\star) \qquad \big|\limsup_{i\rightarrow \infty} B_i^h\big|= 1\,,
$$
then
$$
\mathcal{H}_h (\limsup_{i\rightarrow \infty} B_i)= \mathcal{H}_h ([0,1]).
$$
\end{corollary}
\begin{proof} Since $|B_i|\rightarrow 0$ as $i \rightarrow \infty$, $h$ is right continuous and $h(0)=0$, we have that
$\limsup_{i\rightarrow \infty} B_i\subset[0,1]$ and $\limsup_{i\rightarrow \infty} B^h_i\subset [0,1]$.

Let $\mathcal{B}^h$ be the set $\mathcal{B}^h=\{x \in \R: x \mod 1 \in \limsup_{i\rightarrow \infty} B_i^h\}$ and analogously denote by $\mathcal{B}$ the periodic extension of $\limsup_{i\rightarrow \infty} B_i$.
Observe that $\mathcal{B}^h$ and $\mathcal{B}$ are both $\limsup$ of balls.

Condition $(\star)$ gives us that $|\mathcal{B}^h \cap J|=|J|$, for each interval $J \subset \R$.

The \emph{MTP} gives us then $\mathcal{H}_h(\mathcal{B} \cap J)=\mathcal{H}_h(J)$, for each interval $J \subset \mathbb{R}$, and, in particular,
$$\mathcal{H}_h(\limsup_{i\to \infty} B_i)=\mathcal{H}_h([0,1])\, .$$

\end{proof}

A fundamental result in diophantine approximation is Dirichlet's theorem, which says that for
any irrational $x\in [0,1]$ there are infinitely many rational numbers $p/q$ satisfying $|x- p/q| < 1/q^2$.
\noindent For a general approximating function $\psi : \mathbb{N} \rightarrow \mathbb{R}^+$, we  define the set of
 $\psi$-approximable numbers as
$$
\mathcal{A}(\psi)=\left\{x\in[0,1];~|x- p/q| < \frac{\psi(q)}{q}~ \textrm{for infinitely many rational numbers}~p/q\right\}.
$$
In particular, observe that  $\mathcal{A}(\psi)$ is a set which can be expressed as a $\limsup$ of balls.

\
		
The following theorem by Khintchine (\cite{Kin}, 1924), provides an elegant result establishing a criterion to determine the Lebesgue measure of the set $\mathcal{A}(\psi)$:

\begin{theorem}[Khintchine]
\begin{equation*}
|\mathcal{A}(\psi)| =
\begin{cases}
0, & \textrm{if }\displaystyle\sum_{q=0}^\infty \psi(q)<\infty,\\
\\
1, & \textrm{if } \displaystyle\sum_{q=0}^\infty \psi(q)=\infty,~ \textrm{and} ~\psi ~\textrm{is monotonic}.
\end{cases}
\end{equation*}
\end{theorem}

In particular, when $\psi(q)=1/(q \ln q)$,  we deduce from Khintchine's theorem that $|\mathcal{A}(\psi)|=0$.

For $\psi(q)=q^{-\alpha} (\alpha >1)$,  one has $|\mathcal{A}(\psi)|=0$ and  Besicovitch (\cite{besi}, 1934)  proved  that the Hausdorff dimension of $\mathcal{A}(\psi)$ is $2/(1+\alpha)$. Independently, this result was also proved, and in fact improved, by Jarn\'ik (\cite {jar}, 1931) in the following beautiful theorem:

\begin{theorem}[Jarn\'{i}k]\label{jarnik}
 Let $h$ be a gauge function such that $h(r)/r\rightarrow \infty$ as $r\rightarrow 0$
and $h(r)/r$ is decreasing for all $r$ sufficiently close to $0$. Then
\begin{equation*}
\mathcal{H}_h (\mathcal{A}(\psi)) =
\begin{cases}
0, & \textrm{if }\sum\limits_{q=0}^\infty q ~h\left(\dfrac{\psi(q)}{q}\right)<\infty\\
\\
\infty, & \textrm{if } \sum\limits_{q=0}^\infty q ~h\left(\dfrac{\psi(q)}{q}\right)=\infty,~ \textrm{and} ~\psi ~\textrm{is monotonic}
\end{cases}
\end{equation*}
\end{theorem}

Note that, if $\psi(q) = q^{-\alpha} ~(\alpha > 1)$, then by choosing $h(r)= r^s, 0<s<1$, not only does Jarn\'ik's theorem implies that
$\mbox{\rm dim}(\mathcal{A}(\psi)) = 2/(1+\alpha)$, but also that its Hausdorff measure $\mathcal{H}_{2/(1+\alpha)}(\mathcal{A}(\psi))=\infty$.

 Notice as well that if we ask for a very high degree of approximation, such as $\psi(q)=2^{-q^\alpha}, \alpha>0$,  an instance of so-called super-Liouville numbers, then by choosing $h(r)=1/(\log 1/r)^s, s>0$, we get that $\mathcal{L}{-}\mbox{\rm dim}(\mathcal{A}(\psi))=2/\alpha$, and its corresponding logarithmic Hausdorff measure $\mathcal{L}_{2/\alpha}(\mathcal{A}(\psi))=\infty$.

Since $\mathcal{A}(\psi)$ is a $\lim \sup$ of balls,  the divergence part in Jarn\'{i}k's theorem, which is the difficult one, could  now be easily inferred from Khintchine's theorem via the \emph{MTP}.

Actually, the original statements of Khintchine and Jarn\'{i}k included some extra assumptions on $\psi$, which  ever since have been removed. See for instance \cite {BB} for an excellent survey on this subject.

\section{Dyadic approximation and Dobi\'nski set}\label{sect:dyadic approx and dobinski}

 A  program similar to that sketched above in Section \ref{sect:MTP} could be carried out in order to study the size of sets of numbers well approximable via dyadic fractions, if a dyadic version of  Kintchine's theorem would hold true. We then could apply the \emph{MTP} theorem, as in the diophantine approximation, to extract information on $h$-Hausdorff dimension from the statement on Lebesgue measure.

A natural approach, see, for instance,  \cite{S}, starts with applying the following version of the Borel-Cantelli lemma, which is given in terms of a generalization of the notion of independence in probability. If $(\Omega, \mathcal{A}, \mu)$ is a probability space, we will say that a sequence of events $(E_n)_{n\ge 1}$ in $\mathcal{A}$ is {\bf quasi-independent} if  there exists a constant $C>0$ such that
\[
\mu (E_n\cap E_m)\leq C \mu(E_n) \mu (E_m),
\]
 for all $n,m \ge 1$ such that $n\neq m$.
\begin{lemma}\label{lema:Borel Cantelli quasi} Let $(\Omega,\mathcal{A}, \mu)$ be a probability space, and let  $(E_n)_{n\ge 1}$, be a sequence of events. Let
$$E_\infty= \lim\sup_{n\rightarrow\infty} E_n=\bigcap\limits_{n \ge 1}\bigcup_{k\ge n} E_k\,. $$
Suppose that
$$
\sum_{n=1}^\infty \mu(E_n) = \infty.
$$
 If the sets $E_n$ are quasi-independent, then $\mu (E_\infty)>0$.
\end{lemma}

Although proofs of this result, which originated with  Chung and Erd\H{o}s, \cite{Chung-Erdos}, and Kochen and Stone, \cite{Kochen-Stone},  can be found in many probability text books, let us mention the one in \cite{S}, where a variant  of Kintchine's theorem is also proved.

In the dyadic context, we want to estimate the size of sets of points which are well approximated by dyadic rationals.

Let $\varphi$ be a function defined on $\mathbb{N}$ such that $0<\varphi(n)<c_0$, for each $n \ge 1$ and  some constant $c_0>0$. Associated to $\varphi$ we define the set
$$
\mathcal{B}(\varphi)=\big\{x\in[0, 1]; ~ |x-P_n(x)|<\frac{\varphi(n)}{2^n},\quad\mbox{for infinitely many $n \ge 1$}\Big\}\,.
$$
Recall that $P_n(x)$ is the closest dyadic number in $\mathcal{Q}_n$ to $x$, and observe that the case $\varphi(n)=2^n 2^{-2^n/k}$ corresponds to the subset  $\mathcal{D}(k)$ of the Dobi\'nski set $\mathcal{D}$.

\medskip

\begin{theo}\label{main theorem}Let $h$ be a gauge function in $[0,1]$ such that $h(x)/x$ is not increasing. Let $\varphi: \mathbb{N}\rightarrow (0,\infty)$ so that
\begin{equation}\label{eq:boundedness conditions}\sup\limits_{n \ge 1} \varphi(n)<+\infty \quad \mbox{and} \quad \sup\limits_{n\ge 1} 2^n h\Big(\dfrac{\varphi(n)}{2^n}\Big)< +\infty\, .\end{equation}

  Then
\begin{equation*}
\mathcal{H}_h\big(\mathcal{B}(\varphi)\big) =
\begin{cases}
0, & \textrm{if }\sum_{n=0}^\infty 2^n ~h\left(\dfrac{\varphi(n)}{2^n}\right)<\infty,\\
\\
\mathcal{H}_h ([0,1]), & \textrm{if } \sum_{q=0}^\infty 2^n ~h\left(\dfrac{\varphi(n)}{2^n}\right)=\infty.
\end{cases}
\end{equation*}

\end{theo}

The proof of the statement about $\mathcal{H}_h(\mathcal{B}_\varphi)=0$ is immediate and holds without the boundedness provisions \eqref{eq:boundedness conditions}.

We precede the proof with the following lemma.

\begin{lem}\label{lemma:preceeding main}
Let $\omega: \mathbb{N}\rightarrow (0,c_0)$ for some $c_0>0$.
Assume that
$$
\sum_{n \ge 1} \omega(n)=+\infty\,.$$
Then
$$|\mathcal{B}(\omega)|=1\,. $$
\end{lem}
\begin{proof}
The Borel set $\mathcal{B}(\omega)$ is a tail event for the filtration generated by the independent identically distributed random variables $\epsilon_j, j \ge 1$, in $[0,1]$. Thus, by Kolmogorov $0{-}1$ law,  $|\mathcal{B}(\omega)|=0$ or $|\mathcal{B}(\omega)|=1$.

We will show next that $|\mathcal{B}(\omega)|>0$, to conclude that  $|\mathcal{B}(\omega)|=1$.

Write $\mathcal{B}(\omega)$ as
$$\mathcal{B}(\omega)=\bigcap\limits_{k \ge 1} \bigcup\limits_{n\ge k} U_n\,,$$
where $$U_n=\bigcup\limits_{0\le j \le 2^n} I_{n,j}\,,$$
and
$$
I_{n,j}=\Big(\frac{j}{2^n} -\frac{\omega(n)}{2^n},\frac{j}{2^n} +\frac{\omega(n)}{2^n}\Big) \cap [0,1]\, , \quad \mbox{for $0\le j \le 2^n$}\, .$$

We claim that the  sets $(U_n)_{n \ge 1}$ are  quasi-independent. This fact, taking into account Lemma \ref{lema:Borel Cantelli quasi}, completes the proof.

The boundedness of the function $\omega$ implies uniformly bounded overlap for the intervals in $U_n$, i.e, there existe a constant $C_0$ (depending only on $c_0$) so that
$$2\sum\limits_{0\le j \le 2^n}{\mbox{\Large $\chi$}}_{I_{n,j}}\le C_0\, , \quad \mbox{for each $n \ge 1$}\,,$$
and thus that
$$\sum\limits_{0\le j \le 2^n} |I_{n,j}|\le C_0 |U_n|\le C_0\sum\limits_{0\le j \le 2^n} |I_{n,j}|\, , \quad \mbox{for each $n \ge 1$}\,,$$
and therefore that
$$
\frac{1}{C_0}\, \omega(n)\le |U_n|\le C_0 \,\omega(n)\, , \quad \mbox{for each $n \ge 1$}\,.$$

Let $m >n \ge 1$.

The number of  points of the form $\{k/2^m,~k=0,\dots,2^m\}$ contained in any interval $I$ of size $|I|>1/2^m$ is at most  $ 2^{m{+}1} |I|$. Therefore, for $n \ge 1$ and $0 \le j \le 2^n$ we have that
$$|I_{n,j}\cap U_m|\le  2^{m{+}1}\, |I_{n,j}| \, \frac{\omega(m)}{2^m}= 2 \, \omega(m)\,  |I_{n,j}|\, .$$

We can then deduce that
\begin{equation*}
\begin{aligned}
|U_n\cap U_m| &= \sum_{j=0}^{2^n}  |I_{n,j}\cap U_m|\le 2 \sum_{j=0}^{2^n} \omega(m) \, |I_{n,j}| =2
 \sum_{j=0}^{2^n}  \frac{\omega(n)}{2^n}\, \omega(m)\\
&= 2 \, \omega(n)\, \omega(m)\le 2 \,C_0^2  \, |U_n|\, |U_m|\, .
\end{aligned}
\end{equation*}
This shows that the sets $U_n$ are quasi-independent.\end{proof}

\

Next, we turn now to the
\noindent \textit{Proof of Theorem {\textrm{\ref{main theorem}}}.}

As mentioned above, the fact that $\mathcal{H}_h\big(\mathcal{B}(\varphi)\big) =0$ when the series converges is immediate and does not require the boundedness conditions \eqref{eq:boundedness conditions}.

\medskip

We assume now that $\sum_{n=0}^\infty 2^n ~h\left(\dfrac{\varphi(n)}{2^n}\right)=\infty$.

Let $\theta$ be the auxiliary function in $\mathbb{N}$ given by
$$\theta(n)=2^n h\Big(\frac{\varphi(n)}{2^n}\Big)\, ,\quad \mbox{for $n \ge 1$}\,.$$

By hypothesis, the function $\theta(n)$ is bounded and $\sum_{n \ge 1} \theta(n)=\infty$. Therefore, by Lemma \ref{lemma:preceeding main} we have that
$$|\mathcal{B}(\theta)|=1\,.$$

Observe that
$$\mathcal{B}(\theta)=\bigcap_{j \ge 1}\bigcup_{n \ge j} U_n^h$$
where
$$
U_n^h=\bigcup_{0\le k \le 2^n}\bigg(\frac{k}{2^n}-h\Big(\frac{\varphi(n)}{2^n}\Big),\,\frac{k}{2^n}-h\Big(\frac{\varphi(n)}{2^n}\Big)\bigg)\cap [0,1]\,,\quad \mbox{for $n \ge 1$}\,.$$

The Mass Transference Principle, Corollary \ref{cor MTP},  gives us then that
$$\mathcal{H}_h(\mathcal{B}(\varphi))=\mathcal{H}_h([0,1])\,.$$\qed

\medskip

\begin{coro}\label{cor:size of dobinski}
The Dobi\'nski set $\mathcal{D}$ satisfies $\mathcal{L}{-}\mbox{\rm dim}(\mathcal{D})=1$ and
$$\mathcal{L}_1(\mathcal{D})=\infty\,.$$
\end{coro}
\begin{proof} Fix $k\ge1$ and consider the subset $\mathcal{D}(k)$ of $\mathcal{D}$.
Define $\varphi(n)=2^n 2^{-2^n/k}$, for $n \ge 1$.

If we take $h(x)=1/\ln(1/x)$, for $x \in (0,1]$,  then Theorem \ref{main theorem} gives us $\mathcal{L}_1(\mathcal{D}(k))=\mathcal{L}_1([0,1])=\infty$, for each $k\ge1$. A fortiori, $\mathcal{L}_1(\mathcal{D})=\infty$.

If for $s>1$, we take $h(x)=1/\ln(1/x)^s$, for $x \in (0,1]$, then from Theorem \ref{main theorem} (or just directly) we get $\mathcal{L}_s(\mathcal{D}(k))=0$, for each $k \ge 1$. We conclude that $\mathcal{L}_s(\mathcal{D})=0$, and, consequently that $\mathcal{L}{-}\mbox{\rm dim}(\mathcal{D})=1$.\end{proof}

\medskip

With similar arguments we get the following corollaries.

\begin{coro} Let $\alpha>0$. Consider the set $\mathcal{E}_\alpha \subset[0, 1]$ consisting of those numbers so that, after the $n$-th digit in its dyadic expansion, there are  at least $n\alpha$ digits equal to $0$, and this for  infinitely many $n$. Then
$$
\mbox{\rm dim}~(\mathcal{E}_\alpha)=\frac{1}{1+\alpha}\,.
$$
Besides  $\mathcal{H}_{1/(1+\alpha)}(\mathcal{E}_\alpha)=\infty$.
\end{coro}%
%

\medskip

\begin{coro}
 \label{coro2}
 Let $\alpha>0$. Consider the set $\mathcal{F}_\alpha \subset[0, 1]$ consisting of those numbers so that, after the $n$-th digit in its dyadic expansion, there are  at least $2^{n\alpha}$ digits equal to $0$, and this for infinitely many $n$. Then
$$
\mathcal{L}{-}\mbox{\rm dim}~(\mathcal{F}_\alpha)=\frac{1}{\alpha}\,.
$$
Besides  $\mathcal{L}_{1/\alpha}(\mathcal{F}_\alpha)=\infty.$
\end{coro}
%

\section{Hausdorff measure of willow sets}\label{sect:Cantor}
In this section we present and prove a result regarding the Hausdorff measure of Cantor-type sets, which we call willow sets,  that could arise in several settings, in particular in the context of dyadic approximations and the Dobi\'nski set.

The following  result (see for instance \cite{matt})  is a standard tool to establish  lower bounds on Hausdorff measures (and dimensions).

\begin{lem}[Frostman Lemma]
\label{lemma:frost}
Let $h$ be a gauge function. Suppose that a set $E\subset \mathbb{R}$ carries a probability measure $\mu$ such that, for any interval $I$,
\begin{equation}
\label{eqn:frost}
\mu(I)\leq c~ h(|I|)
\end{equation}
for some constant $c>0$. Then $\mathcal{H}_h(E)>0$.
\end{lem}
In particular,  if $h(r)=r^s$ or $h(r)={1}/{(\log 1/x)^s}$ and $E$ carries such a probability measure, then we would get that $\mbox{\rm dim}(E)\geq s$ or $\mathcal{L}{-}\mbox{\rm dim}(E)\geq s$ respectively.

\medskip

\subsection{Willow sets}We proceed next to describe the construction of willow sets.

A \textit{generation} is a finite collection of disjoint closed intervals in $[0,1]$. Generations are composed of \textit{families}. The intervals of a given family of a given generation have the same length.

We will have generations $\mathfrak{G}_k, k \ge 0$. Generation $0$ consists of just one family with just one interval, namely,  $[0,1]$.

Generation $\mathfrak{G}_k$ consists of $M_k$ families $\mathfrak{F}_{k,j}, 1\le j \le M_k$:
$$
\mathfrak{G}_k= \bigcup_{j=1}^{M_k}\mathfrak{F}_{k,j}\, , \quad \mbox{for $k \ge 0$}\,.$$
Every interval of generation $k$ is required to be contained in one (unique) interval of generation $k{-}1$ called its progenitor.

The common size of all the intervals of the family $\mathfrak{F}_{k,j}$ is denoted by $A({k,j})$; the indexing of the families of generation $k$ is such that $A({k,1}) \ge A({k,2}) \ge \dots \ge A({k,M_k})$.

For $k\ge 0$, we denote by $\mathcal{W}_k$ the union of all  the intervals of generation $k$. Observe that $\mathcal{W}_0=[0,1]$ and that $\mathcal{W}_k \subset \mathcal{W}_{k{-}1}$, for each $k \ge 1$.

\medskip

Our willow set is defined as
$$\mathcal{W}=\bigcap_{k \ge 1}\mathcal{W}_k\, .$$

\medskip

Next we spell out a number of requirements on the sizes of the intervals of families and generations and their distribution that willow sets must satisfy.

\medskip

\noindent A) The number $M_k$ of families of generation $k$ is so large that:
\begin{equation}\label{Mk}
M_k~|J|>1 \textrm{ for any interval}~ J\in \mathfrak{G}_{k-1}\, , \quad \mbox{for $k \ge 1$}\,.
\end{equation}
 Observe that for \eqref{Mk} to hold, it suffices that $M_k$ is so large that
$$
M_k A(k{-}1, M_{k{-}1}) >1\, , \quad \mbox{for $k \ge 1$}\,.
$$

\medskip

\noindent B) Denote the number of intervals in the family $\mathfrak{F}_{k,j}$ which are contained in a given interval  $J\in\mathfrak{G}_{k-1}$ by
   $N_{k,j}(J)$. Although this number  may depend  upon $J$, it is assumed  that there exists a function $g(j, k)$ such that
   \begin{equation}\label{equi}
   \frac{1}{c}~|J|~g(j, k)\leq N_{k,j}(J)\leq c~|J| ~g(j, k)
   \end{equation}
for some constant $c>0$.

\medskip

\noindent C) \textit{Equidistribution property}. It is required  that the intervals belonging to any given family are equidistributed in their  progenitor.

More precisely, we assume that for any interval $I\in [0,1]$, the number of intervals in the family $\mathfrak{F}_{k,j}$ that intersect $I$ is at most $c~|I| ~g(j, k)$;

\medskip

\noindent D) For some $c>0$ and for any  $k \ge 1$, the distance between any two intervals $I_1$ and $I_2$  that belong to $\mathfrak{G}_k$ satisfies
\begin{equation}\label{dist}
\textrm {dist} (I_1,I_2)\geq c~ \min_{I\in \mathfrak{G}_k } |I|=c~A(k, M_k).
\end{equation}

\

\subsection{Hausdorff measure of willow sets} With these requirements in the construction of willow sets $\mathcal{W}$ we have:

\begin{theo}\label{hausdorff of willow}
Let $h$ be a gauge function such that $h(x)/x$ is not increasing for $x$ small enough. With the notation above, suppose that there exists a constant  $c>0$ such that  for all $k \ge 1$ and $1 \le j \le M_k$,
\begin{equation}
\label{eqn:primos:geq}
g(j, k)\geq\frac{c}{h(A(k,j))},
\end{equation}
Then
\[
\mathcal{H}_h(\mathcal{W})>0,
\]
\end{theo}

\begin{proof}

 In order to apply Frostman Lemma, we construct  an appropriate  probability measure $\mu$ on $\mathcal{W}$.

   Let $\mu_0$ be the Lebesgue measure in $\mathcal{W}_0=[0,1]$. In a recursive way, define  for any $k \in \mathbb{N}$ the probability measure $\mu_k$ supported   on $\mathcal{W}_k$, in such a way that
the measure of an interval $J$ of generation $k{-}1$ is distributed among  its descendants $\{L; L \subset J ~\textrm{with}~ L\in \mathfrak{G}_k\}$  according to the formula
$$
\mu_k(L)=\frac{\mu_{k-1}(J)}{M_k~ N_{k, j}(J)}\,, \quad  \mbox{if $L \in \mathcal{F}_{k,j}$}\, .
$$
Besides, in each such interval $L$, the measure $\mu_k$ is a multiple of Lebesgue measure restricted to that interval.

The measures $(\mu_k)_{k\in\mathbb{N}}$ converge weakly to a probability measure $\mu$ supported on $\mathcal{W}$ that coincides with $\mu_k$ on each interval of generation $k$.

\

 We claim that condition \eqref{eqn:frost} of Frostman's Lemma holds for any interval $I \subset [0,1]$

 Consider first the case in which $I$ is actually an interval of a certain generation $k$. Then $I \in \mathfrak{F}_{k,j}$, for some $j$ such that $1 \le j \le M_k$ and so the length of $I$ is $A_{k,j}$. Let $J\in \mathfrak{G}_{k{-}1}$ be its progenitor. By \eqref{equi}, \eqref{Mk} and \eqref{eqn:primos:geq}, we get that
\begin{equation}
\label{eqn:primos:k}
\mu(I)=\frac{\mu_{k-1}(J)}{M_k~ N_k^j(J)}\lesssim \frac{1}{M_k ~|J|~ g(j, k)}
\lesssim  h(A(k,j))\simeq  h(|I|).
\end{equation}

Next  we consider a  general interval $I \subset [0, 1]$. Choose $k$ such that
\[
A(k,M_k)<|I|<A(k{-}1, M_{k{-}1}).
\]
Because of ~\eqref{dist}, the interval $I$ intersects at most one interval $J$ of generation $k{-}1$, and eventually some of the  descendants $L\in \mathfrak{G}_{k}$ of $J$.

We split the proof that \eqref{eqn:frost} holds for this interval $I$  into two cases:

\noindent $i)$ \quad Suppose that $I$ is contained in one  interval $\tilde{I}_k$ of generation $k$. Thus, by \eqref{eqn:primos:k} we already know that $\mu(\tilde{I}_k) \lesssim h(\tilde{I}_k)$.

   To estimate $\mu(I)$ we need to consider the family of intervals  $\mathcal{F}$ which consists  of all the intervals $L$ of generation $k{+}1$ such that  both $L\subset \tilde{I}_k$ and that  $L \cap I\neq \emptyset$.

     Recall that  by  the equidistribution property, the number of intervals in a given family $\mathfrak{F}_{l,k+1}$ that intersect $I$ is bounded by $ g(l,k+1) |I|$. Therefore,
\[
\mu(I)\leq\sum_{L\in \mathcal{F}}\mu_{k+1}(L)\le \sum_{l=1}^{M_{k+1}}\frac{\mu_k(\tilde{I}_k)}{M_{k+1}~ N_{k+1,l}(\tilde{I}_k)}~g(l, k+1)~ |I|
\]
On the other hand,  by \eqref{equi}, $N_{k+1, l}(\tilde{I}_k)\simeq g(l,k+1) |\tilde{I}_k|$, thus
\[
\mu(I)\lesssim \mu_k(\tilde{I}_k)\frac{|I|}{|\tilde{I}_k|}
\lesssim \frac {h(|\tilde{I}_k|)}{|\tilde{I}_k|}|I|\lesssim h(|I|)
\]
since $h(x)/x$ is not increasing sufficiently close to $0$.

\medskip

\noindent $ii)$ \quad Suppose that $I$ is not contained in any interval of generation $k$. In any case, interval $I$ intersects at most one interval $J$ of generation $k{-}1$, as we have already pointed out. But in this case, we can assume without loss of generality that $I$ is contained in $J$, and repeat the argument above replacing the index $k$ with $k{-}1$.

\end{proof}

\medskip

\subsection{Application of willow sets to the Dobi\'nski set} Next, we shall apply Theorem \ref{hausdorff of willow} to give an alternative  proof of Corollary \ref{cor:size of dobinski}.

The idea is to construct an appropriate willow set $\mathcal{W}\subset \mathcal{D}$ and apply Theorem \ref{hausdorff of willow}. As we will observe later, $\mathcal{W}$ is actually contained in the set $\mathcal{T}$ described in the introduction, which consists of those points in $\mathcal{D}(1)$ with long strings of zeros.

For any $k$ we pick two natural numbers $M_k$ and $n_k$ that will be determined later on.

For any $j=1,\dots,M_k$, the family $\mathfrak{F}_{k,j}$ of  generation $k$ consists of   those \lq\lq maximal\rq\rq\,  intervals with left extreme at the dyadic points $\{m/2^{n_k+j}\}_{m=0}^{n_k+j-1}$ and length
\begin{equation}
\label{eqn:dob:length}
A(k,j)=2^{-(n_k+j+2^{n_k+j})}\,, \qquad \mbox{for $j=1,\dots,M_k$}\,,
\end{equation}
that are contained in an interval of generation $k-1$. By \lq\lq maximal\rq\rq intervals we mean that at each dyadic point under consideration we associate the interval with the biggest length, within the generation.
 Observe that the requirements  (A)-(D) hold:

\begin{description}
\item[(A)] Given $\mathcal{W}_{k-1}$,  we choose $M_k$ so that \eqref{Mk} holds.
\item[(D)] Next, we choose $n_k$ large enough so that there are intervals of generation $k$ inside any interval of the previous generation. For instance, $n_k>> 2^{n_{k-1}+M_{k-1}}$. Besides,  \eqref{dist} holds if the length of the largest interval of generation $k$ fills a (uniformly) bounded fraction of the distance between two consecutive dyadic numbers in $\mathcal{Q}_{n_k+M_k}$, that is
\[
2^{-(n_k+M_k)}-A(k, 1)\ge c~A(k, M_k)
\]
By the choice we made of the lengths $A(k, j),$ this condition  can be written as
\[
2^{-(n_k+M_k)}-2^{-(n_k+1+2^{n_k+1})}\ge c~2^{-(n_k+M_k+2^{n_k+M_k})}.
\]
Recall that $M_k$ at this stage is fixed, hence any large enough $n_k$ will work.
\item[(B)-(C)] Each interval in $\mathcal{F}_{k, j}$ which is contained in an interval $J$ in $\mathfrak{G}_{k-1}$ has a left end-point  in a dyadic number $l/2^{n_k+j}$, $l=0,\dots, 2^{n_k+j}-1$, with the additional observation that $l$ must be odd, whenever $j>1$. Therefore
\[
N_{k, j}(J)\sim2^{n_k+j}|J|.
\]
which proves (B) with $g(j,k)=2^{n_k+j}$. A similar argument shows that (C) holds as well.

\end{description}
The set $\mathcal{W}=\bigcap_{k\ge1}\mathcal{W}_k$ constructed in this fashion is contained in $\mathcal{D}(1)$, and therefore in $\mathcal{D}$.

Choosing $h(x)={1}/{\log(1/x)}$ we get that  $g(j, k)\thicksim 1/h(A(k,j))$, and  by Theorem \ref{hausdorff of willow} we can conclude that $\mathcal{L}_1(\mathcal{D})>0$ and therefore  $\mathcal{L}{-}\mbox{\rm dim} (\mathcal{D})\geq 1$.

As shown in the previous section,  $\mathcal{L}{-}\mbox{\rm dim} (\mathcal{D})\leq 1$. Hence, we may conclude that $\mathcal{L}{-}\mbox{\rm dim} (\mathcal{D})=1$. It turns out that once we know that the corresponding Hausdorff measure is positive,  due to the translation invariance property of the set, we get that its measure $\mathcal{L}_1(\mathcal{D})=\infty$, see Lemma 2.1 in \cite {olsen}.

\section{Logarithmic capacity of the Dobi\'nski set}\label{sect:capacity}

There are many connections between the logarithmic capacity of a set and its Hausdorff dimension.

In particular, if the dimension of the set is positive, then its logarithmic capacity is positive as well. On the other hand,  Erd\"{o}s and Gillis (\cite{ErdosGillis}, 1937)  have shown that if  $E\subset \mathbb{R}$ is a compact set with $\mathcal{L}_1 (E) < \infty$ then $\textrm{log cap}(E)=0$. The converse does not hold, in fact there are examples of sets $E$, all satisfying that   $\mathcal{L}_1(E) = \infty$, but with arbitrary logarithmic capacity.

The Dobi\'nski set $\mathcal{D}$ has $\mathcal{L}_1(\mathcal{D}) = \infty$; naturally, we ask: is  $\textrm{log cap}(\mathcal{D})=0$?
Or, on the contrary, are there closed subsets of the Dobi\'nski set with positive logarithmic capacity?

\

\small

\noindent \textit{Alberto Dayan:} Department of Mathematics and Statistics, Washington University in St Louis, USA.
\texttt{alberto.dayan@wustl.edu}

\medskip

\noindent \textit{Jos\'e L. Fern\'andez:} Departamento de Matem\'aticas, Universidad Aut\'onoma de Madrid, Spain.
\texttt{joseluis.fernandez@uam.es}

\medskip

\noindent \textit{Mar\'ia J. Gonz\'alez:} Departamento de Matem\'aticas, Universidad de C\'adiz, Spain.\newline
\texttt{majose.gonzalez@uca.es}

\end{document}